%% file: multiplantlotsizingcomplexity.tex
\documentclass[11pt]{article}
\usepackage{a4wide}
\usepackage{amsmath}
\usepackage{amssymb,amsthm}
\usepackage{fullpage}
\usepackage[utf8]{inputenc}
\usepackage{enumerate}
\usepackage{color}
\usepackage{apacite}
\usepackage{latexsym,epsfig,graphicx}
\usepackage[normalem]{ulem}

\newtheorem{Prop}{Proposition}

\newtheorem{corollary}[Prop]{Corollary}
\newtheorem{Th}[Prop]{Theorem}

\title{On the computational complexity of uncapacitated multi-plant lot-sizing problems}

\author{Jesus O. Cunha {\thanks{Universidade Federal do Ceará, Departamento de Estatística e Matemática Aplicada, Fortaleza, Brazil.  ({\tt jesus.ossian@dema.ufc.br})}} \and Hugo H. Kramer {\thanks{Universidade Federal da Paraíba, Departamento de Engenharia de Produ{\c c}{\~a}o, João Pessoa, Brazil.  ({\tt hkramer@ct.ufpb.br})}} \and Rafael A. Melo {\thanks{Universidade Federal da Bahia, Departamento de Ci\^{e}ncia da Computa\c{c}\~{a}o, Computational Intelligence and Optimization Research Lab (CInO), Salvador, Brazil. ({\tt melo@dcc.ufba.br})}} 
}

\usepackage{epsfig,lscape}
\usepackage{subfigure}

\begin{document}

\maketitle

\begin{abstract}

Production and inventory planning have become crucial and challenging in nowadays competitive industrial and commercial sectors, especially when multiple plants or warehouses are involved.
In this context, this paper addresses the complexity of uncapacitated multi-plant lot-sizing problems. We consider a multi-item uncapacitated  multi-plant lot-sizing problem with fixed transfer costs and show that two of its very restricted special cases are already NP-hard. 
Namely, we show that the single-item uncapacitated multi-plant lot-sizing problem with a single period and the multi-item uncapacitated two-plant lot-sizing problem with fixed transfer costs are NP-hard. Furthermore, as a direct implication of the proven results, we also show that a two-echelon multi-item lot-sizing with joint setup costs on transportation is NP-hard.\newline 


\end{abstract}

\input{01_introduction.tex}

\input{02_problemdescription.tex}

\input{03_UMPLSP_1,NP,NT_.tex}

\input{04_UMPLSP_NI,2,NT_.tex}

\input{05_finalcomments.tex}

\vspace{0.8cm}

{
\noindent \small 
\textbf{Acknowledgments:}
Work of Rafael A. Melo was supported by the State of Bahia Research Foundation (FAPESB); and the Brazilian National Council for Scientific and Technological Development (CNPq).
}

\bibliography{multiplantlotsizingcomplexity}

\bibliographystyle{apacite}

\end{document}

%% file: 01_introduction.tex
\section{Introduction}

Multi-plant problems are of very practical interest as they often arise in the structure of modern supply chains. The multi-item uncapacitated  multi-plant lot-sizing problem with fixed transfer costs consists in determining a production, storage and transfer planning to meet time varying deterministic demand requirements of multiple items for each of the plants on a multi-period planning horizon while minimizing the total cost. 
We believe this problem attracts interest from both theoretical and practical perspectives due to its simple structure and fundamental characteristics.
In this paper we analyze the complexity of the multi-item uncapacitated  multi-plant lot-sizing problem with fixed transfer costs and show that two of its very restricted special cases are already NP-hard.

Certain simple well-known uncapacitated production planning problems were shown to be polynomially solvable, such as the uncapacitated economic lot-sizing problem~\cite{WagWhi58} and some of its multi-level variants~\cite{Zan69,MelWol10}. The economic lot-sizing problem, however, becomes NP-hard when time dependent capacities exist~\cite{BitYan82}.
On the other hand, other production planning problems are already NP-hard even in their uncapacitated variants, such as the joint-replenishment problem and the one-warehouse multi-retailer problem~\cite{ArkJonRou89, CunMel16}.
Furthermore, other NP-hard uncapacitated production planning problems but with additional restrictions which imply that one cannot really produce an unlimited amount include the uncapacitated economic lot-sizing with remanufacturing~\cite{HelJanHeuWag14,CunMel16b} and the uncapacitated lot-sizing with inventory bounds \cite{AkbPenRap15,MelRib17}. The reader is referred to \citeA{BraAbsDauNor17,PocWol06} for reviews on production planning problems.

Several works considered multi-plant lot-sizing problems with inter-plant transfers.
\citeA{Sam95} studied uncapacitated and capacitated lot-sizing problems in a multi-plant, multi-item, multi-period environment in which inter-plant transfers are allowed. The authors have shown that both problems are NP-hard. We remark, however, that the author's work is not readily available for an interested reader. \citeA{SamSch02} considered the capacitated problem studied in \citeA{Sam95} and proposed a heuristic procedure which starts with a solution for the uncapacitated variant and uses a smoothing procedure to remove capacity violations. \citeA{NasResTol10} proposed a greedy randomized adaptive search procedure (GRASP) metaheuristic with path-relinking to find cost-effective solutions for the same problem. 
\citeA{CarNas16} considered the same problem and proposed a Lagrangian heuristic which was able to outperform both a standard formulation implemented in a commercial solver and the GRASP metaheuristic of \citeA{NasResTol10}.
\citeA{NasYanCar18} compared several mixed integer programming formulations for the very same problem. 
\citeA{CarNas18} proposed a kernel search matheuristic for a multi-plant, multi-item, multi-period capacitated lot-sizing problem with inter-plant transfers and setup carry-over. 

\citeA{DarLarCoe16} studied a real case which consists of a multi-plant production planning and distribution problem for the simultaneous optimization of production, inventory control, demand allocation and distribution decisions for delivery within promised time windows, and proposed integer programming formulations for the problem. 
Other more general production and distribution problems involving multiple plants in which the demands are associated to retailers were studied in~\citeA{Par05,MelWol12}.

The main contribution of this paper is to establish the complexity of certain uncapacitated multi-plant lot-sizing problems with inter-plant transfers which, to the best of our knowledge, are not publicly available in the literature. In this way, we show that the single-item uncapacitated multi-plant lot-sizing problem with a single period and the multi-item uncapacitated two-plant lot-sizing problem with fixed transfer costs are NP-hard. Besides, as a consequence of the proven results, a two-echelon multi-item lot-sizing with joint setup costs on transportation is also shown to be NP-hard.

The remainder of this paper is organized as follows. Section~\ref{sec:STDformulation} formally defines the multi-item uncapacitated multi-plant lot-sizing problem with fixed transfer costs. In the later sections, we show that two of its special cases are already NP-hard. Section~\ref{sec:singleitem} shows that the single-item uncapacitated multi-plant lot-sizing problem with a single period is NP-hard. Section~\ref{sec:twoplant} demonstrates that the multi-item uncapacitated two-plant lot-sizing with fixed transfer costs is NP-hard, which leads to showing that the two-echelon multi-item lot-sizing with joint setup costs
on transportation is NP-hard. Concluding remarks are presented in Section~\ref{sec:finalcomments}.

%% file: 02_problemdescription.tex
\section{The multi-item uncapacitated multi-plant lot-sizing problem with fixed transfer costs} \label{sec:STDformulation}

The multi-item uncapacitated multi-plant lot-sizing problem with fixed transfer costs (MIUMPLS) can be formally defined as follows.
Define $I = \{1,\dots,NI\}$ to be a set of items, $T = \{1,\dots,NT\}$ to be a set of time periods, and $P = \{1,\dots,NP\}$ to be a set of production plants. 
Let $d_{t}^{ip}$ be a dynamic deterministic demand for each item $i \in I$ in each plant $p \in P$ that must be met at time period $t \in T$, considering that no backorders are allowed. 
Consider fixed setup costs and unitary production costs which are incurred to produce an item $i \in I$ in plant $p \in P$ at time period $t\in T$ ($f_t^{ip}$ and $c_t^{ip}$, respectively). The unitary inventory holding cost of an item $i \in I$ in plant $p \in P$ at time period $t \in T$ is given by $h_{t}^{ip}$. Unitary transfer cost of an item $i \in I$ from plant $p \in P$ to plant $l \in P$, with $p \neq l$, at time period $t\in T$ is given by $r^{ipl}_t$. Additionally, a transfer fixed cost $F^{pl}_t$ is incurred whenever items are transferred from plant $p \in P$ to plant $l \in P$, with $p \neq l$, at time period $t$. The problem thus consists in determining a production/transfer plan aiming to minimize the total sum of setup, production, transfer, and inventory costs. It is assumed that there is no initial storage.

In order to formally describe the problem as a mixed integer program, define variables $x_{t}^{ip}$ to be the amount produced of item $i \in I$ in plant $p \in P$ at time period $t \in T$. Define variables $s_{t}^{ip}$ to be the amount of item $i \in I$ held in storage in plant $p \in P$ at the end of time period $t \in T$. Let variable $w_{t}^{ip l}$ define the amount of item $i \in I$ transferred from plant $p \in P$ to plant $l \in P$, with $p \neq l$, at time period $t \in T$. {Also, it is assumed that such transfer starts and ends at time period $t \in T$.} Consider the binary variable $y_{t}^{ip}$ defining whether the production of item $i \in I$ occurs in plant $p \in P$ at time period $t \in T$. Lastly, define the binary variable $Y^{pl}_t$ to be equal to one if there is transfer from plant $p \in P$ to plant $l \in P$, with $p \neq l$, at time period $t \in T$ and zero otherwise. Furthermore, let $M$ be a sufficiently large number. Using these
defined variables, the problem can be cast as the following mixed integer program
\begin{align}
z_{MIUMPLS}=
\min &  \sum_{i \in I} \sum_{p \in P} \sum_{t \in T} \left( c_t^{ip} x_{t}^{ip} + f_t^{ip} y_{t}^{ip} + h_{t}^{ip} s_{t}^{ip} + \sum_{\substack{l \in P,\\l \neq p}} r_t^{ipl} w_{t}^{ipl} \right) + \sum_{p \in P}\sum_{\substack{l \in P,\\l \neq p}}\sum_{t \in T} F^{pl}_t Y^{pl}_t \label{STD:FO} \\
& s_{t-1}^{ip} + x_{t}^{ip} + \sum_{\substack{l \in P, \\l \neq p}} w_{t}^{ilp}  = s_{t}^{ip} + \sum_{\substack{l \in P, \\l \neq p}} w_{t}^{ipl} + d_{t}^{ip}, \ \textrm{ for} \ i \in I, \ p \in P, \ t \in T, \label{STD:invbal} \\
& x_{t}^{ip} \leq M y_{t}^{ip}, \ \textrm{ for}\ i \in I, \ p \in P, \ t \in T, \label{STD:disj} \\
& w_t^{ipl} \leq M Y_t^{pl},\ \textrm{ for}\ i \in I,\ p \in P,\ l \in P,\ t \in T,\ \mbox{with } p \neq l, \label{STD:cap} \\
& x_{t}^{ip}, s_{t}^{ip}  \geq 0,\ \textrm{ for}\ i \in I,\ p \in P,\ t \in T, \label{STD:xsdef} \\
& w_{t}^{ipl} \geq 0,\ \textrm{ for}\ i \in I,\ p \in P,\ l \in P,\ t \in T,\ \mbox{with } p \neq l, \label{STD:rdef} \\ 
& y_{t}^{ip} \in \{0, 1\},\ \textrm{ for}\ i \in I,\ p \in P,\ t \in T, \label{STD:ydef} \\
& Y_t^{pl} \in \{0, 1\},\ \textrm{ for}\ p \in P,\ l \in P,\ t \in T,\ \mbox{with } p \neq l. \label{STD:Ydef}
\end{align}
The objective function~(\ref{STD:FO}) minimizes the total cost, which is composed of setup, production, storage and transfer costs. Constraints~(\ref{STD:invbal}) are inventory balance constraints and ensure that the demands are met without backordering. Constraints~(\ref{STD:disj}) and (\ref{STD:cap}) enforce the setup variables to one whenever, respectively, production and transfer occur. Constraints~(\ref{STD:xsdef})-(\ref{STD:Ydef}) define the nonnegativity and integrality requirements on the variables.

The decision version of the multi-item uncapacitated multi-plant lot-sizing problem with fixed transfer costs asks whether there is a solution whose cost is at most $K$, where $K$ is given as input.

%% file: 03_UMPLSP_1,NP,NT_.tex
\section{NP-hardness of the single-item uncapacitated multi-plant lot-sizing problem with a single period}
\label{sec:singleitem}

In this section, we show that a special case of the multi-item uncapacitated multi-plant lot-sizing problem with fixed transfer costs (MIUMPLS), namely, the single-item uncapacitated multi-plant lot-sizing problem (UMPLS) with a single period, is NP-hard. This is achieved using a reduction from the decision version of the uncapacitated facility location problem (UFL).

The uncapacitated facility location problem can be formally defined as follows. Consider a set $S=\{1,\ldots,NS\}$ of potential facility locations, a set $C=\{1,\ldots,NC\}$ of clients, a fixed cost $q_j$ to open facility $j\in S$, and a cost $v_{lj}$ of serving client $l \in C$ from facility $j \in S$. The problem consists in obtaining a subset $S'\subseteq S$ of the facilities to be opened and then to assign clients to these facilities while minimizing the total cost, given as ${z}_{UFL} = \sum_{j \in S'}q_j + \sum_{j \in S'}\sum_{l \in C_j}v_{lj}$, where $C_j$ represents the set of clients served by facility $j$.




The decision version of the uncapacitated facility location asks whether there is a subset $S'\subseteq S$ that can be opened and the clients can be assigned to its facilities such that the total cost is at most $K''$, where $K''$ is given as input.

\begin{Th}\label{complexityUMLS}
The single-item uncapacitated multi-plant lot-sizing problem with a single period is NP-hard.
\end{Th}
\begin{proof}
In the following, a polynomial transformation UFL{$\propto$}UMPLS is presented.
Given an instance for the uncapacitated facility location, an instance for the single-item uncapacitated multi-plant lot-sizing will be constructed as follows.
Define a unique item and a single period, i.e., $NI = NT = 1$. Create a set $P=\{1, \ldots, NS, NS+1,\ldots,NS+NC \}$ of plants, corresponding to every facility in $S=\{ 1, \ldots, NS \}$ and every client in $C=\{ 1, \ldots, NC \}$.  
For every plant associated with a facility $j\in S$, set $d^{1j}_1 = 0$, $f^{1j}_1=q_j$ and $c^{1j}_1=0$.
For every plant associated with a client $l \in C$, set $d^{1,NS+l}_1 = 1$, and $f^{1,NS+l}_1 = c^{1,NS+l}_1 = \Omega$, where $\Omega$ represents a very large value which is polynomially bounded by the input size such as $(\max_{j\in S}q_j + \max_{l\in C, j\in S}v_{lj})\times (NC+1)$. 
For each pair $\{j,l\}$, such that $j \in S$ and $l\in C$, set $r^{1j,NS+l}_1 = v_{lj}$ and $r^{1,NS+l,j}_1 = {\Omega}$.
For each ordered pair $(j,j')$, such that $j,j' \in S$, with $j\neq j'$, set $r^{1jj'}_1 = {\Omega}$ and for each ordered pair $(l,l')$, such that $l,l'  \in C$, with $l\neq l'$, set $r^{1,NS+l,NS+l'}_1 = {\Omega}$.
We now show that the UFL instance has a solution with value at most $K''$ iff the UMPLS instance has a solution with value at most $K''$. Firstly, consider a solution $S'$ for the UFL instance, in which each facility $j\in S'$ covers a set of clients $C_j$, with cost $K''$ . Let the corresponding solution for UMPLS be $(\hat{y},\hat{x},\hat{s},\hat{w})$ in which $\hat{y}^{1j}_1 = 1$ and  $\hat{x}^{1j}_1 = \sum_{l \in C_j}d^{1l}_1$ for $j \in S'$, $\hat{w}^{1jl}_1 = 1$ for each ordered pair $(j,l)$ such that $j \in S'$ and $l \in C_j$, and all other variables are set to zero. 
Therefore, the cost of $(\hat{y},\hat{x},\hat{s},\hat{w})$ is $\hat{z}_{UMPLS} = \sum_{j \in S'} \left(f_1^{1j} \hat{y}_{1}^{1j} + \sum_{\substack{l \in C_j}} r_1^{1j,NS+l} \hat{w}_{1}^{1j,NS+l} \right) = \sum_{j \in S'} \left(q_j + \sum_{\substack{l \in C_j}} v_{lj} \right) = K''$.
On the other hand, let $(\hat{y},\hat{x},\hat{s},\hat{w})$ be an integral solution for the UMPLS instance with cost $K$. Note that for any solution in which there are nonintegral $w$ variables, there exists an integral solution at least as good as this nonintegral one. We construct a solution for UFL as follows. Define $S' = \{ j \in S \ | \ \hat{y}_{1}^{1j} = 1 \}$ and, for each $j\in S'$ build the set $C_j = \{l \in C \ | \ \hat{w}^{1j,NS+l}_1 = 1  \}$. The cost of such solution is thus $\hat{z}_{UFL} = \sum_{j \in S'}q_j + \sum_{j \in S'}\sum_{l \in C_j}v_{lj} = \sum_{j \in S'} \left(f_1^{1j} \hat{y}_{1}^{1j} + \sum_{\substack{l \in C_j}} r_1^{1j,NS+l} \hat{w}_{1}^{1j,NS+l} \right) = K$. Thus, the result holds.  
\end{proof}

%% file: 04_UMPLSP_NI,2,NT_.tex
\section{NP-hardness of the multi-item uncapacitated two-plant lot-sizing problem with fixed transfer costs}
\label{sec:twoplant}

In this section, we show that another special case of the multi-item uncapacitated multi-plant lot-sizing problem with fixed transfer costs is NP-hard. More specifically,
we prove that the decision version of the joint-replenishment problem (JRP) can be reduced to the decision version of the multi-item uncapacitated two-plant lot-sizing problem with fixed transfer costs (MIU2PLS), implying that the optimization version of the latter is NP-hard.

The joint-replenishment problem can be formally defined as follows. Consider a set $I=\{1,\ldots,NI\}$ of items, a set $T=\{1,\ldots,NT\}$ of periods, a demand $d'^i_{t}$ associated to every item $i\in I$ and period $t\in T$, a per item setup cost $f'^i_{t}$ for producing item  $i \in I$ in period $t\in T$, a joint setup cost $F'_t$ if any item is produced in period $t$, a per unit production cost $c'^i_{t}$ for $i\in I$ and $t\in T$, and inventory holding costs $h'^i_t$ for $i\in I$ and $t\in T$. The problem consists in determining a production plan minimizing the total sum of production and inventory costs.

In order to formally describe the problem as a mixed integer program, define variable $x'^{i}_{t}$ to be the amount of item $i\in I$ produced in period $t\in T$, $y'^{i}_{t}$ to be equal to one if production of item $i\in I$ occurs in period $t\in T$ and to be equal to zero otherwise, $Y'_t$ to be equal to one if production of any item occurs in period $t\in T$ and to be equal to zero otherwise, and $s'^{i}_{t}$ to be the amount of item $i\in I$ held in storage at time period $t\in T$. Additionally, let $M$ be a sufficiently large number. The joint-replenishment problem can be formulated as the mixed integer program

\begin{align}
z_{JRP} = \min & \sum_{i \in I} \sum_{t \in T} \left( c'^{i}_t x'^{i}_{t} + f'^{i}_t y'^{i}_{t} + h'^{i}_{t} s'^{i}_{t} \right) + \sum_{t \in T} F'_t Y'_t \label{JRP:FO} \\
& s'^{i}_{t-1} + x'^{i}_{t} = s'^{i}_{t} + d'^{i}_{t},\ \textrm{ for} \ i \in I, \ t \in T, \label{JRP:invbal} \\
& x'^{i}_{t} \leq M y'^{i}_{t},\ \textrm{ for}\ i \in I,\ t \in T, \label{JRP:disj} \\
& y'^{i}_{t} \leq Y'_t,\ \textrm{ for} \ i \in I,\ t \in T, \label{JRP:cap} \\
& x'^{i}_{t}, s'^{i}_{t} \geq 0,\ \textrm{ for} \ i \in I, \ t \in T, \label{JRP:xsdef} \\
& y'^{i}_{t} \in \{0, 1\},\ \textrm{ for} \ i \in I, \ t \in T, \label{JRP:ydef}\\
& Y'_t \in \{0, 1\},\ \textrm{ for} \ t \in T. \label{JRP:Ydef}
\end{align}
The objective function (\ref{JRP:FO}) minimizes the total production and inventory costs. Constraints~(\ref{JRP:invbal}) are balance constraints. Constraints~(\ref{JRP:disj}) enforce the per item setup variables to one whenever a given item is produced while constraints~(\ref{JRP:cap}) imply that the joint setup variables are set to one whenever any production occurs. Constraints~(\ref{JRP:xsdef})-(\ref{JRP:Ydef}) guarantee the nonnegativity and integrality requirements on the variables. 
The decision version of JRP consists in determining whether there is a solution with total cost at most $K'$.  

\newpage

\begin{Th}\label{th:miu2pnphard}
The multi-item uncapacitated two-plant lot-sizing problem with fixed transfer costs is NP-hard.
\end{Th}

\begin{proof}
In what follows, a polynomial transformation JRP{$\propto$}MIU2PLS is described. Given an instance for the joint-replenishment problem, we construct an instance for the multi-item uncapacitated two-plant lot-sizing problem with fixed transfer costs as follows. Each item $i\in I$ of JRP is associated to an item in MIU2PLSP.  Each time period $t\in T$ of JRP is linked to a time period of MIU2PLSP.  
For plant $p=1$, set $d^{i1}_t=0$, $f_t^{i1} = f'^{i}_{t}$,  $ c_t^{i1}= c'^{i}_{t}$, $r_t^{i12}= 0$, $h_{t}^{i1} = {\Omega}$ for every $i\in I$ and $t\in T$, and  $F^{12}_t = F'_t$ for every $t\in T$, considering {$\Omega$} to represent a very large number which is polynomially bounded by the input size such as $(\max_{i\in I, t \in T} f'^{i}_{t} + \max_{i\in I, t \in T} c'^{i}_{t} + \max_{i\in I, t \in T} h'^{i}_{t} + \max_{t \in T} F'_t)\times (NT+1)$.
For plant $p=2$, set $d^{i2}_t= d'{_{t}^{i}}$, $f_t^{i2} = c_t^{i2}= r_t^{i21} = {\Omega}$, $h_{t}^{i2} = h'^{i}_{t}$ for every $i\in I$ and $t\in T$, and $F^{21}_t = {\Omega}$ for every $t\in T$.
We now show that the JRP instance has a solution with value at most $K'$ iff the MIU2PLS instance has a solution with value at most $K'$. Considering a feasible solution $(\hat{y}',\hat{Y}',\hat{x}',\hat{s}')$ for the JRP instance, we build a solution $(\hat{y},\hat{x},\hat{s},\hat{w})$ for the MIU2PLS instance as follows. For every item $i\in I$ and every period $t\in T$, set 
$\hat{y}^{i1}_t = \hat{y}'{^{i}_{t}}$, 
$\hat{x}^{i1}_t = \hat{x}'{^{i}_{t}}$, 
$\hat{y}^{i2}_t = 0$, 
$\hat{x}^{i2}_t = 0$,
$\hat{s}^{i1}_t = 0$,
$\hat{s}^{i2}_t = \hat{s}'{^{i}_{t}}$,
$\hat{w}^{i12}_t = \hat{x}'{^{i}_{t}}$, 
$\hat{w}^{i21}_t = 0$. Additionally, for every period $t\in T$, set $\hat{Y}^{12}_t = \hat{y}'_t$ and $\hat{Y}^{21}_t = 0$.
The cost of this solution is 
\begin{flalign*}
\hat{z}_{MIU2PLS} = & \sum_{i \in I} \sum_{p \in P} \sum_{t \in T} \left( c_t^{ip} \hat{x}_{t}^{ip} + f_t^{ip} \hat{y}_{t}^{ip} + h_{t}^{ip} \hat{s}_{t}^{ip} + \sum_{\substack{l \in P, \\l \neq p}} r_t^{ipl} \hat{w}_{t}^{ipl} \right) + \sum_{p \in P}\sum_{\substack{l \in P, \\l \neq p}}\sum_{t \in T} F^{pl}_t \hat{Y}^{pl}_t\\
= & \sum_{i \in I} \sum_{t \in T} \left( c_t^{i1} \hat{x}_{t}^{i1} + f_t^{i1} \hat{y}_{t}^{i1} + h_{t}^{i2} \hat{s}_{t}^{i2} \right) + \sum_{t \in T} F^{12}_t \hat{Y}^{12}_t\\
= & \sum_{i \in I} \sum_{t \in T} \left( c'^{i}_{t} \hat{x}'{_{t}^{i}} + f'{_{t}^{i}} \hat{y}'{_{t}^{i}} + h'{_{t}^{i}} \hat{s}'{_{t}^{i}} \right) + \sum_{t \in T} F'_t \hat{Y}'_t\\
= & K'.
\end{flalign*}
Now, considering a feasible solution $(\hat{y},\hat{Y},\hat{x},\hat{s},\hat{w})$ for the MIU2PLS instance with finite objective value $K$, we build a solution $(\hat{y}',\hat{Y}',\hat{x}',\hat{s}')$ for the JRP instance as follows. For every item $i\in I$ and period $t\in T$, let $\hat{y}'{^{i}_{t}} = \hat{y}^{i1}_t$, $\hat{x}'{^{i}_{t}} = \hat{x}^{i1}_t$, $\hat{s}'{^{i}_{t}} = \hat{s}^{i2}_t$. Furthermore, for each period $t\in T$, let $\hat{y}'_t = \hat{Y}^{12}_t$. The cost of such solution is
\begin{flalign*}
\hat{z}_{JRP} = & \sum_{i \in I} \sum_{t \in T} \left( c'^{i}_{t} \hat{x}'{_{t}^{i}} + f'{_{t}^{i}} \hat{y}'{_{t}^{i}} + h'{_{t}^{i}} \hat{s}'{_{t}^{i}} \right) + \sum_{t \in T} F'_t \hat{Y}'_t\\
= & \sum_{i \in I} \sum_{t \in T} \left( c_t^{i1} \hat{x}_{t}^{i1} + f_t^{i1} \hat{y}_{t}^{i1} + h_{t}^{i2} \hat{s}_{t}^{i2}  \right) + \sum_{t \in T} F^{12}_t \hat{Y}^{12}_t\\
= & K.
\end{flalign*}
Therefore, the result holds.
\end{proof}

The reduction shown in Theorem~\ref{th:miu2pnphard} implies the following corollary.
\begin{corollary}\label{cor:twoechelon}
The two-echelon multi-item lot-sizing problem with joint setup costs on transportation is NP-hard.
\end{corollary}

We remark that Corollary~\ref{cor:twoechelon} complements the result which shows that the two-echelon multi-item lot-sizing problem with joint setup costs on production is NP-hard, as it was shown to be equivalent to the one-warehouse multi-retailer problem in~\citeA{CunMel16}.

%% file: 05_finalcomments.tex
\section{Concluding remarks}
\label{sec:finalcomments}

In this paper, we investigated the computational complexity of uncapacitated multi-plant lot-sizing problems which present simple structure and yet contain fundamental characteristics of practical supply chains. 
We have shown that the single-item uncapacitated multi-plant lot-sizing problem with a single period is NP-hard. Furthermore, we have proved that the multi-item uncapacitated two-plant lot-sizing problem with fixed transfer costs is NP-hard. As a direct consequence of the presented proofs, we have also demonstrated that the two-echelon multi-item lot-sizing with joint setup costs on transportation is NP-hard.
